\documentclass[11pt]{article}
\usepackage[english]{babel}
\usepackage[T1]{fontenc}
\usepackage[latin1]{inputenc}

\usepackage{amsfonts,amssymb,amsmath,amsthm}
\usepackage{bbm}
\usepackage{xcolor}
\usepackage{graphics}
\usepackage{graphicx}
\usepackage{mathtools}
\usepackage{hyperref} 
\usepackage{soul}

\newtheorem{theorem}{Theorem}[section]
\newtheorem{lemme}[theorem]{Lemma}

\theoremstyle{remark}

\numberwithin{equation}{section}

\title{Biparental model with infinite selection}
\author{Camille Coron, Yves Le Jan}
\date{}

\usepackage{color}
\usepackage{amsmath}
\usepackage{amsthm}

\title{Genetic contribution of an advantaged mutant in the biparental Moran model}
\author{Camille Coron, Yves Le Jan}
\date{}

\usepackage{color}
\usepackage{amsmath}
\usepackage{amsthm}

\begin{document}
\maketitle

\section{Model}

We consider a population of haploid individuals reproducing sexually, i.e. for which the genome of each individual is a random mixture of the genome of its two parents. We assume that initially one individual carries a mutation at one locus, and that individuals carrying this mutation have an advantage regarding genome transmission. Our aim is to study the long time effect of this mutation on the genetic composition of the population, when population size is large.

Biparental genealogies have received some interest, notably in \cite{Chang1999,Derrida2000,GravelSteel2015}, in which time to recent common ancestors and ancestors' weights are investigated for the Wright-Fisher biparental model. In \cite{geneal}, we studied the asymptotic law of the  contribution of an ancestor, to the genome of the present time population. The articles \cite{MatsenEvans2008} and \cite{BartonEtheridge2011} study the link between pedigree, individual reproductive success and genetic contribution. 

As in the previous paper \cite{geneal}, this population is assumed to be composed of a fixed number $N$ of individuals, and is modeled using a Moran biparental model. However this model is modified in order to take into account the advantage conferred by the considered mutation. The monoparental Moran model with selection at birth has received some interest, notably in \cite{EtheridgeGriffiths} that studies its dual coalescent and \cite{KluthBaake} that notably studies alleles fixation probabilities and ancestral lines. Here we consider a biparental model, which means that each individual has two parents. Selection can therefore influence either the reproduction or the death of individuals, but here we assume that selection occurs at death and has a strong effect. More precisely, at each discrete time step, two individuals are chosen independently and uniformly to be parents, and one individual, chosen uniformly but only among non advantaged individuals (i.e. individuals that do not carry this mutation), is then replaced by the offspring of the two parents. This is a way to define a strong selection mechanism at death, in the sense that the ratio of the death rates of disadvantaged and advantaged individuals is infinite. Note in particular that in this model, advantaged individuals cannot die. Our aim here is to study the impact of the invasion of this strong positive mutation on the population genetic composition, so we do not consider the process once all individuals carry the same advantageous allele.

Genetic transmission is assumed to follow Mendel rules, which means that for a given locus, one of the two parents is chosen uniformly (among the two) and transmits its allele (i.e. a copy of its genome at this locus) to the offspring. This transmission is not independent for loci that are close on the genome, but can be considered as independent for loci that are on different chromosomes for instance. By convention, we call "mother" the parent that transmits its allele at the locus under mutation, and "father" the other parent. To study the genetic composition of the population, we now consider a new locus, distant enough from the one presenting the mutation, so that the genome is assumed to be transmitted independently at these two loci. Our interest is then to study the probability for a given allele at this new locus, to come from the individual that was originally advantaged.

Individuals are assumed to live in $N$ sites labeled by $\{1,...,N\}$ and we can set that the initial mutant lives in site $1$. Let $\mu_n$, $\pi_n$ and $\kappa_n$ be respectively the sites of the mother, father and replaced individual, at time $n\in\mathbb{Z}_+$. Let also $\mathcal{Y}_n\subset \{1,...,N\}$ be the set of advantaged individuals at time $n$. The population dynamics is characterized by the stochastic process $(\mu_n,\pi_n,\kappa_n,\mathcal{Y}_n)_{n\in\mathbb{Z}_+}$, whose natural filtration is denoted by $(\mathcal{F}_n)_{n\in\mathbb{Z}_+}$.

For any $i,j\in\{1,...,N\}$, we denote by $A_n(i,j)$ the probability that the allele at the new locus, carried by the individual at site $i$ living at time $n\in\mathbb{Z}_+$ comes from the ancestor living at site $j$ at time $0$, conditionnaly to the filtration $\mathcal{F}_n$. The sequence of matrices $(A_n)_{n\in\mathbb{Z}_+}$ is defined recursively as follows :

$$A_0 \text{ is the identity matrix.}$$
and for any $n\in\mathbb{Z}_+$,
$$\left\{\begin{array}{l}
 A_{n+1}(\kappa_n,j)=\frac{1}{2}(A_n(\mu_n,j)+A_n(\pi_n,j)) \\ 
A_{n+1}(i,j)=A_n(i,j) \text{ for $i\neq \kappa_n$.}
\end{array}\right.$$

The quantity
$$M_n(j)=\sum_{i=1}^N A_n(i,j)$$
 is called the genetic weight of ancestor $j$ after $n$ time steps: it is equal to $N$ times the probability that an allele sampled at time $n$ (far enough on the genome from the locus under selection) comes from the individual living at site $j$ at time $0$.

The number $(|\mathcal{Y}_n|)_{n\in\mathbb{Z}_+}=(Y_n)_{n\in\mathbb{Z}_+}$ of advantaged individuals is a non decreasing Markov chain, that jumps from $k$ to $k+1$ with probability $k/N$ at each time step, or stays in $k$. Let us define $$S_k=\inf\{n:Y_n=k\}$$ It is important to note that knowing that $S_k>n-1$, then $\{S_k=n\}$ if and only if the mother $\mu_{n-1}$ belongs to $\mathcal{Y}_{n-1}$, which occurs with conditional probability $Y_{n-1}/N$. 

The following theorem gives the order of magnitude of the weight $M_{S_N}(1)$ of the unique initially advantaged individual, when population size goes to infinity, once all individuals are advantaged.

\begin{theorem} \label{thm}The weight $M_{S_N}(1))$ of the initially advantaged individual $1$ satisfies $$\mathbb{E}(M_{S_N}(1)))\underset{N\rightarrow+\infty}{\sim} \frac{4}{\sqrt{\pi}} \sqrt{N}.$$
\end{theorem}

Note that in the absence of selection, at each time, the expectation of the weight $M_{\infty}(j)$, for any $j$, is equal to $1$, by exchangeability between individuals. In the absence of selection, the law of the asymptotic weight of a set of ancestors was studied in \cite{geneal}. Note also that in the present case with infinite selection, the weight of each individual which is not the initially advantaged one, is also of order $1$.

\begin{proof}
Let us define respectively the weight of the initially advantaged individual, among advantaged and non advantaged individuals: $$B_n=\sum_{i\in\mathcal{Y}_n} A_n(i,1)$$ and
$$C_n=\sum_{i\notin\mathcal{Y}_n} A_n(i,1).$$
The model early defined leads to the following set of equations:
\begin{equation}\label{eq:B-i}\mathbb{E}(B_{n+1}|Y_{n+1}=Y_n,\mathcal{F}_n)=B_n.\end{equation}
\begin{equation}\label{eq:B-i+1}\mathbb{E}(B_{n+1}|Y_{n+1}=Y_{n}+1,\mathcal{F}_n)=B_n+\frac{B_n}{2Y_n}+\frac{1}{2N}(B_n+C_n).\end{equation}
\begin{equation}\label{eq:C-i}\mathbb{E}(C_{n+1}|Y_{n+1}=Y_n)=C_n-\frac{C_n}{N-Y_n}+\frac{C_n}{2(N-Y_n)}+\frac{1}{2}\frac{B_n+C_n}{N}\end{equation}
\begin{equation}\label{eq:C-i+1}\mathbb{E}(C_{n+1}|Y_{n+1}=Y_n+1)=C_n-\frac{C_n}{N-Y_n}\end{equation}
Equation \eqref{eq:B-i} is explained as follows :
\begin{align*}\mathbb{E}(B_{n+1}|Y_{n+1}=Y_n,\mathcal{F}_n)&=\mathbb{E}\left(\sum_{i\in\mathcal{Y}_{n+1}} A_{n+1}(i,1)|Y_{n+1}=Y_n,\mathcal{F}_n\right)\\&=\mathbb{E}\left(\sum_{i\in\mathcal{Y}_{n+1}} A_{n}(i,1)|Y_{n+1}=Y_n,\mathcal{F}_n\right)\end{align*} since $i\neq \kappa_n$ if $i\in\mathcal{Y}_{n+1}$ and $Y_{n+1}=Y_n$.

In Equation \eqref{eq:B-i+1}, the first term comes from advantaged individuals at time $n$, the second term is the weight brought by the mother (which is necessarily advantaged since $Y_{n+1}=Y_n+1$) and the last term is the weight brought by the father. 

In Equation \eqref{eq:C-i}, the first term comes from disadvantaged individuals at time $n$, the second term comes from the weight loss due to the death of a disadvantaged individual, the third term comes from the weight of the mother (which is necessarily disadvantaged since $Y_{n+1}=Y_n$), and the fourth term comes from the weight of the father.

In Equation \eqref{eq:C-i+1}, the first term comes from disadvantaged individuals at time $n$, and the second term comes from the weight loss due to the death of a disadvantaged individual.

Therefore if we set
$$u_k=\mathbb{E}(B_{S_k}) \text { and } v_k=\mathbb{E}(C_{S_k})$$ then

$$ \left( \begin{tabular}{c}
    $u_{k+1}$ \\
    $v_{k+1}$ 
\end{tabular}\right)=\mathcal{A}_k \left( \begin{tabular}{c}
    $u_{k}$ \\
    $v_{k}$ 
\end{tabular}\right)$$
where 
$$\mathcal{A}_k=\sum_{l=0}^{+\infty}L_k (H_k)^l=L_k[I-H_k]^{-1}$$
and
$$H_k=\left(1-\frac{k}{N}\right)\left(\begin{tabular}{cc}
  $1$   & $0$ \\
  $\frac{1}{2N}$   &  $1-\frac{1}{2(N-k)}+\frac{1}{2N}$
\end{tabular}\right)$$
$$L_k=\frac{k}{N}\left(\begin{tabular}{cc}
  $1+\frac{1}{2k}+\frac{1}{2N}$   & $\frac{1}{2N}$\\
  $0$    &  $1-\frac{1}{N-k}$
\end{tabular}\right).$$

For any triangular matrix $$\left(\begin{tabular}{cc}
  $1-a$   & $0$ \\
  $b$   &  $1-c$
\end{tabular}\right)^{-1}=\left(\begin{tabular}{cc}
  $\frac{1}{1-c}$   & $0$ \\
  $-\frac{b}{(1-a)(1-c)}$   &  $\frac{1}{1-a}$
\end{tabular}\right).$$ 
Hence, after a few simplifications,
 $$[I-H_k]^{-1}= \frac{N}{k} \left( \begin{tabular}{cc}
  $1$   & $0$ \\
  $\frac{N-k}{(2N+1)k}$   &  $\frac{2N}{2N+1}$
\end{tabular}\right).$$
Therefore
$$\mathcal{A}_k=\left(\begin{tabular}{cc}
  $1+\frac{1}{2k}+\frac{1}{2N}+\frac{N-k}{N}\frac{1}{2k(2N+1)}$   & $\frac{1}{2N+1}$ \\
  $\frac{N-k-1}{(2N+1)k}$   &  $\left(1-\frac{1}{N-k}\right)\left(1-\frac{1}{2N+1}\right)$
\end{tabular}\right).$$

Let us set $\tilde{u}_k=\frac{u_k}{k}$ and $\tilde{v}_k=\frac{v_k}{N-k}$ for $k\geq 1$. Then\\

\begin{equation}\label{rec-uv} \left( \begin{tabular}{c}
    $\tilde{u}_{k+1}$ \\
    $\tilde{v}_{k+1}$ 
\end{tabular}\right)=\tilde{\mathcal{A}}_k \left( \begin{tabular}{c}
    $\tilde{u}_{k}$ \\
    $\tilde{v}_{k}$ 
\end{tabular}\right)\end{equation}

with \begin{equation}\label{eq:A}\tilde{\mathcal{A}}_k=\left(\begin{tabular}{cc}
  $1-\frac{1}{2N+1}\left(1-\frac{1}{2k+1}+\frac{1}{2N}-\frac{1}{2N(k+1)}\right)$   & $\frac{N-k}{(2N+1)(k+1)}$ \\
  $\frac{1}{(2N+1)k}$   &  $1-\frac{1}{2N+1}$
\end{tabular}\right).\end{equation}

 Remarkably enough, note that

\begin{equation}\label{rec-x}\tilde{u}_{k+1}-\tilde{v}_{k+1}=\frac{2Nk+N+k}{(2N+1)(k+1)}(\tilde{u}_{k}-\tilde{v}_{k}).\end{equation} Theorem \ref{thm} is then given by the following lemma.  \end{proof}

\medskip
\begin{lemme}\label{lem-u}
$$u_N\underset{N\rightarrow \infty}{\sim}\frac{4}{\sqrt{\pi}}\sqrt{N}$$
\end{lemme}

\begin{proof} Let us define $x_{k}=\tilde{u}_{k}-\tilde{v}_{k}$ for all $k\in\{1,...,N\}$. 
Since $u_1=1$ and $v_1=0$, from Equation \eqref{rec-x} we have that for any $k\geq 0$,
$$x_{k+1}=\prod_{l=1}^k\frac{2Nl+N+l}{(2N+1)(l+1)}=\prod_{l=1}^k\left(1-\frac{1}{2N+1}\right)\left[1-\frac{1}{2(l+1)}+\frac{1}{2N}\left(1-\frac{1}{l+1}\right)\right].$$
Therefore 
\begin{align*} x_{k+1}&=\left(1-\frac{1}{2N+1}\right)^k\prod_{l=1}^k\frac{2l+1}{2l+2}\prod_{l=1}^k\left[1+\frac{1}{2N}\frac{2l}{2l+1}\right]\\&=\prod_{l=1}^k\left[1-\frac{1}{(2l+1)(2N+1)}\right]\prod_{l=1}^k\frac{2l+1}{2l+2}\\&=\prod_{l=1}^k \left[1-\frac{1}{(2l+1)(2N+1)}\right]\frac{2(2k+2)!}{4^{k+1}((k+1)!)^2}\end{align*}
By Stirling formula:
$$ \frac{2(2k+2)!}{4^{k+1}((k+1)!)^2}= \frac{2\sqrt{2\pi (2k+2)}\left(\frac{2k+2}{e}\right)^{2k+2}}{4^{k+1}\times 2\pi (k+1)\left(\frac{k+1}{e}\right)^{2k+2}}\left(1+O\left(\frac{1}{k}\right)\right)=\frac{2}{\sqrt{\pi (k+1)}}\left(1+O\left(\frac{1}{k}\right)\right).$$
Using that
 $\prod_{l=1}^k\left[1-\frac{1}{(2l+1)(2N+1)}\right]\geq1-\sum_{l=1}^k\frac{1}{(2l+1)(2N+1)}>1-\frac{\log(2k+2)}{2(2N+1)}$
we have that there exists a positive constant $C$ such that for any $k\geq 1$
$$\frac{2}{\sqrt{\pi k}}\left(1-\frac{C(\log(k)+1)}{N}\right)\leq x_{k}\leq \frac{2}{\sqrt{\pi k}}\left(1+\frac{C}{k}\right).$$ 

Besides, $\tilde{u}_N=x_N+\tilde{v}_N$ and from Equations \eqref{rec-uv} and \eqref{eq:A}, we have

$$\tilde{v}_{k+1}=\frac{\tilde{u}_k}{2N+1}+\frac{2N}{2N+1}\tilde{v}_k=\frac{x_k}{2N+1}+\tilde{v}_k$$
which gives that
$$\tilde{v}_N=\frac{1}{2N+1}\sum_{l=1}^{N-1}x_{l}
.$$

Therefore:
$$\frac{1}{2N+1}\sum_{k=1}^{N-1}\frac{2}{\sqrt{\pi k}}\left(1-\frac{C(\log(k)+1)}{N}\right)\leq \tilde{v}_N\leq \frac{1}{2N+1}\sum_{k=1}^{N-1}\frac{2}{\sqrt{\pi k}}\left(1+\frac{C}{k}\right).$$
Now $\int_1^N\frac{2}{\sqrt{\pi x}}dx<\sum_{k=1}^{N-1}\frac{2}{\sqrt{\pi k}}<\int_0^{N-1}\frac{2}{\sqrt{\pi x}}dx$. Therefore $\sum_{k=1}^{N-1}\frac{2}{\sqrt{\pi k}}\underset{N\rightarrow +\infty}{\sim}\frac{4\sqrt{N}}{\sqrt{\pi}} $.\\
Moreover, $\sum_{k=1}^{N-1}\frac{\log(k)+1}{N\sqrt{\pi k}}$ is bounded (and converges to 0 if $N\rightarrow \infty$).\\
 Hence, $\tilde{v}_N\underset{N\rightarrow +\infty}{\sim} \frac{2}{\sqrt{\pi N}}$. Since $x_N \underset{N\rightarrow +\infty}{\sim} \frac{2}{\sqrt{\pi N}}$, $\tilde{u_N}\underset{N\rightarrow +\infty}{\sim} \frac{4}{\sqrt{\pi N}}.$

\end{proof}

\bibliographystyle{abbrv}
\bibliography{biblio}

\end{document}